\documentclass{amsart}
\usepackage{amsmath}
\usepackage{amsfonts}
\usepackage{amssymb,amscd,amsthm}
\usepackage{geometry}
\usepackage{mathrsfs}
\usepackage[bookmarksnumbered, colorlinks, linktocpage, plainpages]{hyperref}
\usepackage{nameref}
\usepackage{wasysym,amssymb,eufrak,indentfirst,color}
\usepackage{enumitem} 
\usepackage{mathtools}
\usepackage{appendix}
\usepackage{extarrows}
\usepackage{setspace}

\usepackage{cases}
\usepackage{graphicx}
\usepackage{ragged2e}

\numberwithin{equation}{section}
\newcommand\asertion[1]{ssertion $({\mathrm{\romannumeral #1\relax}})$}

\newcommand{\ass}{associator}

\newcommand{\almost}{para-}

\newcommand{\Hom}{\text{Hom}}
\newcommand{\End}{\text{End}}

\newcommand{\spc}{\mathbb{C}}
\newcommand{\spr}{\mathbb{R}}
\newcommand{\spo}{\mathbb{O}}

\newcommand\huaa[1]{\mathscr{A}(#1)}
\newcommand\hua[3]{\mathscr{#1}^{#2}(#3)}
\newcommand{\re}{\text{Re}\,}%

\newcommand\fx[2]{\left<#1,#2\right>}%

\newcommand\fsh[1]{\left|\left|#1\right|\right|}%

\def\O{\mathbb{O}}
\def\R{\mathbb{R}}

\def\abs#1{\left|#1\right|}

\newcommand\clifd[1]{C\ell_{#1}}

\newtheorem{mydef}{Definition}[section]
\newtheorem{rem}[mydef]{Remark}
\newtheorem{eg}[mydef]{Example}
\newtheorem{cor}[mydef]{Corollary}

\newtheorem{lemma}[mydef]{Lemma}
\newtheorem{thm}[mydef]{Theorem}

\newtheorem{step }[stp]{Step }

\begin{document}
\title[Structure of octonionic Hilbert spaces]{ Structure  of  octonionic Hilbert spaces with applications in  the Parseval   equality and   Cayley-Dickson algebras }	
\author{Qinghai Huo}
\email[Q.~Huo]{hqh86@mail.ustc.edu.cn}
\address{Department of Mathematics, University of Science and Technology of China, Hefei 230026, China}

\author{Guangbin Ren}
\email[G.~Ren]{rengb@ustc.edu.cn}
\address{Department of Mathematics, University of Science and Technology of China, Hefei 230026, China}
\date{}
\keywords{Octonionic Hilbert space; sedenions; Cayley-Dickson algebras;
	tensor product; Parseval's   equality}
	\subjclass[2010]{Primary: 17A60;46S10}
	
	\thanks{This work was supported by the NNSF of China (11771412).}
	
	\begin{abstract} Contrary to the simple  structure  of the  tensor product of the quaternionic Hilbert space,  the octonionic situation  becomes more involved.
It turns out that an octonionic Hilbert space  can be decomposed as     an orthogonal direct sum of two subspaces, each of them  isomorphic to a tensor product of an irreducible octonionic Hilbert space with a real Hilbert space.
 As an application,
 we find that  for a given  orthogonal basis   the octonionic Parseval equality holds
 if and only if the basis  is   weak associative.
  Fortunately, there   always exists a  weak associative orthogonal basis
  in an  octonionic Hilbert space.  This completely removes the obstacles
 caused by the failure of the octonionic Parseval equality.
 	As another application,	we provide a new approach to study
the   Cayley-Dickson algebras, which turn out to be   specific  examples of octonionic Hilbert spaces.  An explicit weak associative orthonormal basis is constructed in  each    Cayley-Dickson algebra.

	\end{abstract}

\maketitle
	
	\tableofcontents
	
	
	\section{Introduction}
	The theory of octonionic Hilbert spaces is initiated  by Goldstine and  Horwitz \cite{goldstine1964hilbert} in 1964. It has many developments  in the spectral theory \cite{ludkovsky2007Spectral},  operator theory \cite{ludkovsky2007algebras},  and mathematical   physics  \cite{deleo1996oqm,gunaydin1973quark,gunaydin1976OHilbert,Rembielinski1978tensorOHilbert}.

However, for the normed division algebras, many results are restricted to  the   quaternionic case and remain  open in the octonionic case \cite{Voight2021}.

A   quaternionic Hilbert space is of  a simple structure:

$\bullet$   A Hilbert  quaternionic bimodule is the tensor product of the quaternion algebra $\mathbb H$   with a real Hilbert space  \cite{ng2007quaternionic}.

$\bullet$
A Hilbert left quaternionic module  is also a tensor product up to an isomorphism. This is because the compatible right multiplication  always exits which  depends on a   choice of a   Hilbert basis  as shown in \cite[Section 3.1]{ghiloni2013slicefct}.

	  A natural question   for the octonionic case is whether     a Hilbert left octonionic module is the tensor product of the  octonionic algebra $\mathbb O$ with some real Hilbert space.
If the answer  was  affirmative, then it would extremely simplify the theory of   Hilbert left octonionic modules. We refer to
 \cite{ludkovsky2007algebras} for such a development of the theory of octonionic Hilbert spaces under the assumption of  such a  tensor product decomposition.

In this article, we give  a negative answer to the question above.  It turns out that  a Hilbert  left octonionic module can be decomposed into an orthogonal sum of two subspaces, each with a tensor product decomposition.


  To state our results in detail, we first  recall the definition of  Hilbert left $\O$-modules due to  Goldstine and Horwitz \cite{goldstine1964hilbert}.

  \begin{mydef}[\cite{goldstine1964hilbert}]\label{def:gold}
	A Hilbert left $\O$-module	  $H$ is a left $\O$-module with an  $\O$-inner product  $$\left\langle\cdot,\cdot \right\rangle :H\times H \rightarrow \mathbb{O}$$
	such that $(H, \re \left\langle\cdot,\cdot \right\rangle)$ is a real Hilbert space. Here the $\O$-inner product satisfies the following axioms for all $ u,v\in H$ and $p\in \O$:
	\begin{enumerate}[label=(\alph*)]
		\item  $\left\langle u+v,w\right\rangle=\left\langle u,w\right\rangle+\left\langle v,w\right\rangle$;
		\item  $\left\langle u ,v\right\rangle=\overline{\left\langle v ,u\right\rangle}$;
		\item  $\left\langle u ,u\right\rangle\in \spr^+$;  and $\left\langle u ,u\right\rangle=0$ if and only if $u=0$;
		\item $\left\langle tu ,v\right\rangle=t\left\langle u ,v\right\rangle$ for $t\in \R$;
		\item $\re \left\langle pu ,v\right\rangle=\re (p\left\langle u ,v\right\rangle)$;
		\item $\left\langle pu ,u\right\rangle=p\left\langle u ,u\right\rangle$.
	\end{enumerate}	
\end{mydef}

Recently, we find that   axiom $(f)$ is   non-independent  \cite{huoqinghai2021Riesz}. Since the functionals induced from the $\mathbb O$-inner product   are not $\mathbb O$-linear, this motivates us to  introduce the para-linearity instead \cite{huoqinghai2021Riesz}.
It turns out that  para-linearity is the main object of   octonionic functional analysis,  other than
  octonionic linearity.

 We first consider the underline module  structure
	of  a Hilbert left $\O$-module $H$.  We refer to \cite{Shestakov1974rightrep,zhevlakov1982Rings} for the structure of alternative modules, which take $\O$-modules as specific examples.

It is known that there are only two  irreducible left $\O$-modules, i.e.,   $\O$ and   $\overline{\O}$. Here  the conjugate regular module $\overline{\O}$ coincides with
		$\mathbb O$ as a set and its
		 left module structure  is   defined by
		$$p\hat{\cdot}x:=\overline{p}x,$$
		for any $p\in \spo$ and   $x\in \spo$.
Moreover,   any left $\O$-module $H$  admits
a decomposition   \cite{huo2021leftmod}
\begin{eqnarray}\label{ eq:dec-mod-282}H=\spo\huaa{H}\oplus {\spo}\hua{A}{-}{H},\end{eqnarray}
where     $$\huaa{H}:=\{m\in H\mid (pq)m-p(qm)=0\text { for all } p,q \in \O\}$$
is the  set of \textbf{associative elements}, called
the \textbf{nucleus}  of  $M$,  and
$$\hua{A}{-}{H}:=\{m\in H\mid (pq)m-q(pm)=0\text { for all } p,q \in \O\}$$ is the  set of \textbf{conjugate associative elements}.
  Here, we denote $$\mathbb OU:=\left\{\sum_{i=1}^Np_iu_i\mid p_i\in \mathbb O, \ u_i\in U,\ N\in \mathbb{N} \right\} $$
   for any    $U\subseteq H$.

 With the $\O$-module structure of $\O$ and $\overline{\O}$ at hands,  we can endow   a natural $\O$-module structure on the set
 \begin{eqnarray}\label{eq:tensor-pr-dec-283}\Big(\O\otimes_\R\huaa{H}\Big)
	\oplus
	\Big(\overline{\O}\otimes_\R\hua{A}{-}{H}\Big)\end{eqnarray}
  The decomposition (\ref{ eq:dec-mod-282})  can be   modified    as
  $$H\cong \Big(\O\otimes_\R\huaa{H}\Big)
	\oplus
	\Big(\overline{\O}\otimes_\R\hua{A}{-}{H}\Big),$$
  which is a natural isomorphism of $\O$-modules.

 Notice that this isomorphism is only at the level of left $\mathbb O$-modules. We need to
seek further an isomorphism at the level of Hilbert left $\mathbb O$-modules.
To this end, we will
	endow the left $\O$-module  in (\ref{eq:tensor-pr-dec-283})
  with a canonical $\O$-inner product.
It is  shown that the submodules in the two summands in
(\ref{eq:tensor-pr-dec-283})
are orthogonal to each other, and
the $\mathbb O$-inner products restricted to both $\huaa{H}$ and  $ \hua{A}{-}{H}$
are all   real-valued, respectively.
This leads to  the tensor product decomposition of Hilbert left $\O$-modules.

\begin{thm}[\textbf{Tensor product decomposition}]{}\label{thm:decomp-tp-281}
	Let $H$ be a Hilbert left $\O$-module. Then  there exists an $\O$-isomorphism of $\O$-Hilbert spaces
	\begin{eqnarray}\label{eq:intro hilbert space}
	H\cong  \Big(\O\otimes_\R\huaa{H}\Big)
\oplus
\Big(\overline{\O}\otimes_\R\hua{A}{-}{H}\Big).
	\end{eqnarray}

\end{thm}

Theorem \ref{thm:decomp-tp-281} can be applied to study   the  Parseval theorem in $\O$-Hilbert spaces. 	As observed by Goldstine and Horwitz  in the appendix of \cite{goldstine1964hilbert2},   the  Parseval equality may fail   for arbitrary orthonormal basis in an $\O$-Hilbert space. We find for a given orthonormal basis, the  related Parseval equality holds
if and only if the orthonormal basis is \textbf{weak associative} (see Definition \ref{def:weak ass}). Fortunately, a  weak associative orthonormal basis   always exists in an $\O$-Hilbert space. This makes the Parseval equality   applicable in the octonionic setting.

Theorem \ref{thm:decomp-tp-281} can also be applied to study the Cayley-Dickson algebra $\mathbb A_n$. We refer to  \cite{Dickson1919cdalg,harvey1990spinors} for the theory of Cayley-Dickson algebras. In this article, we introduce a new approach to study
 the Cayley-Dickson algebra $\mathbb A_n$. That is, we regard  $\mathbb A_n$ as a Hilbert left $\O$-module when $n\geqslant 4$.  It leads to   an isomorphism of Hilbert left $\O$-modules  $$\mathbb{A}_n\cong \O^{2^{n-4}}\oplus \overline{\O}^{2^{n-4}}.$$
In particular, for the algebra of sedenions $\mathbb S$, we have
   $$\mathbb S=\mathbb O\oplus  \overline{\mathbb O}.$$
 We can construct an explicit weak associative orthonormal basis in $\mathbb A_n$. {Note that a Cayley-Dickson algebra is generally not a quaternionic vector space. The fact that
 $\mathbb A_n$ is  a Hilbert left $\O$-module
  provides a new method to study Cayley-Dickson algebras.}

	\section{Preliminaries}\label{sec:preliminary}
	In this section we review 	
	some definitions and basic properties about the  algebra $\O$ of octonions,  $\O$-modules  and $\O$-Hilbert spaces.

	\subsection{Octonions}\label{subsec:O}
	The algebra $\spo$ of octonions     is the $8$-dimensional  non-associative, non-commutative, normed division algebra over  $\spr$. For convenience, we  denote  $ e_0=1$. Let     $e_0,e_1,\dots,e_7$ be a basis of $\O$ as a real vector space subject to  the multiplication rule
		\begin{align}\label{eq:epsilon notation}
	&e_ie_j=\epsilon_{ijk}e_k-\delta_{ij}, \quad i,j=1,\dots,7,
	\end{align}
	where $\delta_{ij}$ is the Kronecker delta  and $\epsilon_{ijk}$ is a completely skew-symmetric  with value 1 precisely when $$ijk = 123, 145, 176, 246, 257, 347, 365.$$



	
	An octonion can be written as $$x=x_0+\sum_{i=1}^7x_ie_i,$$
	where $ x_i\in\spr$ for $i=0,\dots,7$.
	Its conjugate   is defined by $$\overline{x}:=x_0-\sum_{i=1}^7x_ie_i,$$  its norm  equals $|x|=\sqrt{x\overline{x}}\in \spr$, and its real part  is $\re{x}=x_0=\frac{1}{2}(x+\overline{x})$.
	We denote by $\mathbb{S}_6$ the set of imaginary units in $\O$, i.e.,
	$$\mathbb{S}_6:=\{J\in \O\mid J^2=-1\}.$$
	Then there is a book structure on octonions (see \cite{Jinming2020slicedirac}):$$\O=\bigcup  \mathbb{H}_{I,J},$$
	where  the sum runs over all $I, J\in \mathbb S_6$ which are mutually orthogonal
	and $\mathbb{H}_{I,J}$ denotes the quaternionic algebra  spanned by
	$\{1,I,J,IJ\}$.
	


	\subsection{$\O$-modules}
 An  octonionic module is   a specific alternative module	since the octonionic algebra  is an alternative algebra. The general representation theory over alternative algebras has been fully studied; see \cite{Schafer1952repaltalg,jacobson1954structure,Shestakov1974rightrep,Shestakov2016bimod,huo2021leftmod}.  We   recall  some basic notations and  results on $\O$-modules in this subsection.




	 \begin{mydef}\label{def: 1 alternative algebra}

	 	An $\R$-vector space $M$ is called a \textbf{left $\O$-module}, if there is an  $\R$-linear map
	 	\begin{eqnarray*}
	 		L:  \O&\rightarrow &\End_{\R}M\\
	 		p&\mapsto& L_p
	 	\end{eqnarray*}
	 	satisfying  $L_1=id_M$
	 	and
	 	\begin{eqnarray}\label{eq:left ass}
	 	[p,q,x]=-[q,p,x],
	 	\end{eqnarray}
	 	for all $ p,q\in \O$ and $x\in M$.
	 	Here  $$[p,q,x]:=(pq)x-p(qx)=L_{pq}(x)-L_pL_q(x)$$ is called the left \ass\ of $M$.
	 	The definition of  right $\O$-module  is similar.

	 \end{mydef}
%
%
%
%
	It is a useful fact that
  \eqref{eq:left ass} is  equivalent to the following statement
	$$r(rm)=r^2m $$
for all 	$r\in \O$ and all $ m\in M$.
%


The  structure of  a left $\O$-module is characterized by its associative elements and conjugate associative elements as follows.  The  \textbf{associative elements}  of a left $\O$-module $M$  form a set $$\huaa{M}:=\{m\in M\mid [p,q,m]=0\text { for all } p,q \in \O\},$$
 called the  \textbf{nucleus} of  $\O$-module $M$.
The \textbf{conjugate  associative elements}  of a left $\O$-module $M$  form a set denoted by  $$\hua{A}{-}{M}:=\{m\in M\mid (pq)m=q(pm)\text { for all } p,q \in \O\}.$$

	\begin{thm}[\cite{huo2021leftmod}]\label{thm:M=OA+OA^-}
		Let $M$ be a left $\O$-module. Then  $$M=\spo\huaa{M}\oplus {\spo}\hua{A}{-}{M}.$$
	\end{thm}

	The category of left $\spo$-modules is shown to be isomorphic to the category of left $\clifd{7}$-modules \cite{huo2021leftmod}. For a left $\O$-module $M$, we can endow it with a natural $\clifd{7}$-module structure, denoted by $M_{\clifd{7}}$. Let $M$, $M'$ be two left $\O$-module. We denote
	$$\Hom_\spo(M,M'):=\{f\in\Hom_{\R}(M,M') \mid f(px)=pf(x) \text{ for all } p\in \O\}.$$
	Then it follows from  Schur's lemma that
	\begin{eqnarray}
\Hom_\spo(\O,\overline{\O})\cong\Hom_{\clifd7}(\O_{\clifd7},\overline{\O}_{\clifd7})=\{0\}.
	\end{eqnarray}
	We restate this fact in a more  precise way as follows.
	
	\begin{lemma}\label{lem:f=0 f(px)=overline(p)x}
		If $f\in End_\spr(\spo)$ satisfies $f(px)=\overline{p}f(x)$ for all $p,x \in \spo$, then $f=0$.
	\end{lemma}
	%
	A simple variant of Lemma \ref{lem:f=0 f(px)=overline(p)x} is as follows.
	\begin{lemma}\label{lem:f=0 f(xq)=qf(x)}
		Let $f\in \End_{\R}(\O)$. If  it holds $f(xq)=qf(x)$ for all $q,x\in \O$, then $f=0$.
	\end{lemma}
	\begin{proof}
		We define $g(x):=f(\overline{x})$. Then we obtain
		$$g(px)=f(\overline{x}\ \overline{p})=\overline{p}f(\overline{x})=\overline{p}g({x}).$$ It thus follows from Lemma \ref{lem:f=0 f(px)=overline(p)x} that $g=0$, i.e., $f=0$.
	\end{proof}
	
	

\subsection{$\O$-Hilbert spaces}
Octonionic Hilbert spaces are introduced   by Goldstine and Horwitz \cite{goldstine1964hilbert}. Recently, we \cite{huoqinghai2021Riesz} develop this theory by using a new notion of $\O$-\almost linearity.

\begin{mydef}[\cite{huoqinghai2021Riesz}] Let $M$ be a left $\O$-module.
	A real linear map $f: M \xlongrightarrow[]{} \O$
	is called an \textbf{$\O$-\almost linear} function if it satisfies
	$$\re [p,x,f]=0$$
	for any $p\in\O$ and $x\in M$.
	Here $$[p,x,f]:=f(px)-pf(x),$$ which is called the \textbf{second associator}.
	
\end{mydef}

The definition of $\O$-Hilbert spaces can be restated in terms of \almost linearity.
\begin{mydef}[\cite{huoqinghai2021Riesz}]\label{def:hilbert O space}
	A left $\spo$-module $H$ is called a \textbf{ pre-Hilbert  $\spo$-module} if there exists an $\spr $-bilinear map $\left\langle\cdot,\cdot \right\rangle :H\times H \rightarrow \O$, which is referred to as an \textbf{$\spo$-inner product}, satisfying:
	\begin{enumerate}
		\item \textbf{($\O$-\almost linearity)} $\left\langle\cdot,u\right\rangle$ is (left) $\spo$-\almost linear for all $u\in H$.
		\item \textbf{(Octonion hermiticity)} $\left\langle u ,v\right\rangle=\overline{\left\langle v ,u\right\rangle}$ for all $ u,v\in H$.
		\item \textbf{(Positivity)} $\left\langle u ,u\right\rangle\in \spr^+$ and $\left\langle u ,u\right\rangle=0$ if and only if $u=0$.
	\end{enumerate}
\end{mydef}
We can define the \textbf{norm} $\fsh{\cdot}:H\rightarrow \mathbb{R}^+ $ by setting
\begin{eqnarray}\label{def:fshu ||u||}
\fsh{u}=\sqrt{\fx{u}{u}}.
\end{eqnarray}
	A pre-Hilbert left $\spo$-module	$H$ is said to be a \textbf{Hilbert left $\spo$-module} if it is complete with respect to its natural distance  induced by the norm.

%
%

\begin{eg}\label{eg:o,o-}
	We give two typical  examples of Hilbert left $\spo$-modules.
	\begin{enumerate}
		\item $\O$ is a Hilbert left $\spo$-module with the $\O$-inner product $\fx{\cdot}{\cdot}_\O$ defined by
		$$\fx{x}{y}_{\O}=x\overline{y}$$for all $x,y\in \O$.
		\item The conjugate regular module $\overline{\O}$ is a Hilbert left $\spo$-module with the $\O$-inner product $\fx{\cdot}{\cdot}_{\overline{\O}}$ defined by
			$$\fx{x}{y}_{\overline{\O}}=y\overline{x}$$for all $x,y\in \O$. We remark that its left module structure    is  defined by
$$p\hat{\cdot}x:=\overline{p}x,$$
for any $p\in \spo$ and   $x\in \overline{\O}$.
		
	\end{enumerate}
\end{eg}

In an $\O$-Hilbert space $H$, we can regard $v\in H$ as an $\mathbb O$-para-linear map induced by the $\mathbb O$-inner product.
This motivates us to introduce the \textbf{second associator} of $H$ as
$$[p,u,v]:=[p,u,  \langle \cdot, v \rangle]
=
\left\langle pu ,v\right\rangle-p\left\langle u ,v\right\rangle$$ for $u,v\in H$ and $p\in \O$,


{
From now on, we  shall denote $$\re\fx{u}{v}:=\fx{u}{v}_\R$$ for all $u,v\in H$.
We assemble some identities concerning  second associators for later use. }

\begin{lemma}[\cite{huoqinghai2021Riesz}]
	For all $u,v\in H$ and all $p,q\in \O$, the following hold
\begin{eqnarray}
	\left\langle [p,q,u],v\right\rangle _\R&=&-\left\langle u,[p,q,v]\right\rangle _\R; \label{eq:<[qpx]y>0=<[qpy]x>}\\
{[p,v,u]}&=&-[p,u,v];\label{eq:Ap(u,v)=-Ap(v,u)}\\
{[pq,v,u]}&=&\fx{[p,q,v]}{u}-[p,q,\fx{v}{u}]+p[q,v,u]+[p,qv,u].\label{eq:Apq(v,u)=<[p,q,v],u>-[p,q,<v,u>]+pAq(v,u)+Ap(qv,u) }
\end{eqnarray}

\end{lemma}

The following characterization of associative elements is very useful in the sequel.
\begin{lemma}[\cite{huoqinghai2021Riesz}]\label{lem:lem57}
	An element $x\in \huaa{H}$  if and only if
	\begin{equation} \label{eq:assup-ass-392}
	[p,x,y]=0
	\end{equation}
	for all $p\in\spo$ and $y\in H.$
\end{lemma}

In contrast to the complex and quaternionic cases,  the $\O$-inner products are $\O$-\almost linear, rather than   $\O$-linear. The following lemma reveals the close relations  among $\spo$-inner product, $\spo$-scalar product,  and the second associator.

\begin{lemma}[\cite{huoqinghai2021Riesz}]\label{lem:the second associator of H 's prop}
	Let $H$ be a pre-Hilbert left $\spo$-module. Then for all $u,v\in H$ and all $p,q\in \spo$, the following  hold:
	\begin{eqnarray}
	\fx{u}{pv}&=&\fx{u}{v}\overline{p}+[p,u,v];\label{eq:<u,pv>=<u,v>p^+Ap(uv)}\\
	\fx{pu}{qv}&=&(p\fx{u}{v})\overline{q}+[{pq},u,v]+\fx{[p,q,v]}{u}.\label{eq:<pu,pv>in lemma}
	\end{eqnarray}
	
\end{lemma}

\section{Tensor product decompositions}
In this section, we provide  a Hilbert left $\O$-module with a  tensor product decomposition structure.
We first define  the tensor product of a Hilbert left $\O$-module with a real Hilbert space.
\begin{mydef}\label{def:tensor}
	Let $(H_1,\fx{\cdot}{\cdot}_1)$ be a Hilbert left $\O$-module and $(H_2,\fx{\cdot}{\cdot}_2)$ be a real Hilbert space. Then we endow the tensor product  $H_1\otimes_\R H_2$  with
	a canonical left $\O$-module structure
	\begin{eqnarray*}
		\O\times \big(H_1\otimes_\R H_2\big)&\to& H_1\otimes_\R H_2\\
		(p,u\otimes_\R x)&\mapsto&(pu)\otimes_\R x
	\end{eqnarray*} and an  $\mathbb O$-inner product    \begin{eqnarray}
	\fx{u\otimes_\R x}{v\otimes_\R y}:=\fx{u}{v}_1\fx{x}{y}_2
	\end{eqnarray}
	for all $u,v\in H_1$ and all $x,y\in H_2$. Then one can check that the tensor product $(H_1\otimes_\R H_2,\fx{\cdot}{\cdot})$ is a Hilbert left $\O$-module.
\end{mydef}
\medskip

\begin{mydef}
	Let $(H_1,\fx{\cdot}{\cdot}_1)$ and $(H_2,\fx{\cdot}{\cdot}_2)$ be two Hilbert left $\O$-modules. We endow the direct sum $H_1\oplus H_2$  with the canonical left $\O$-module structure and the $\O$-inner product	$$\fx{(x,u)}{(y,v)}:=\fx{x}{y}_1+\fx{u}{v}_2$$ for all $x,y\in H_1$ and $u,v\in H_2$. One can check that the direct sum $(H_1\oplus  H_2,\fx{\cdot}{\cdot})$ is a Hilbert left $\O$-module. By definition we know that  $H_1$ is orthogonal to $H_2$ if we regard $H_1$ and $H_2$ as canonical  subsets of $H_1\oplus H_2$.
\end{mydef}

\begin{mydef}
	Let $(H_1,\fx{\cdot}{\cdot}_1)$ and $(H_2,\fx{\cdot}{\cdot}_2)$ be two Hilbert left $\O$-modules. A map $$f:H_1\to H_2$$ is called an isomorphism  of Hilbert left $\O$-modules if
	\begin{enumerate}
		\item  for all $p\in \O$ and $x\in H_1$, $$f(px)=pf(x);$$
		\item for all  $x,y\in H_1$, $$\fx{f(x)}{f(y)}_2=\fx{x}{y}_1.$$
	\end{enumerate}
\end{mydef}
 It is  useful  to know   the behavior of the $\mathbb O$-inner product  on associative and conjugate associative elements   in an $\spo$-Hilbert  space.

\begin{lemma}\label{cor:ass and conj ass in H}
	Let $H$ be a Hilbert left $\O$-module. Then  we have
	\begin{enumerate}
		\item for  all $u,v\in \huaa{ H}$,
		\begin{align}\label{eq:<uv>R}
		\fx{u}{v}&\in\spr;
		\end{align}
		\item for all $u\in \huaa{H}$ and $v\in \hua{A}{-}{H}$, 	\begin{align}\label{eq:<uv>=0}
		\fx{u}{v}=0.
		\end{align}
	\end{enumerate}
\end{lemma}
\begin{proof}
	We first prove a\asertion{1}.
	Recall identity \eqref{eq:Apq(v,u)=<[p,q,v],u>-[p,q,<v,u>]+pAq(v,u)+Ap(qv,u) }:
	$${[pq,v,u]}=\fx{[p,q,v]}{u}-[p,q,\fx{v}{u}]+p[q,v,u]+[p,qv,u]$$
	for all $p,\, q\in \O$  and $u,v\in H$. If $u,\,v\in \huaa{H}$, then it follows from Lemma \ref{lem:lem57}
	that
	$$[p,{q},\fx{v}{u}]=0$$
	for all $p,\, q\in \O$.
	This implies  $\fx{v}{u}\in \mathbb R$  as desired.
	
	Now we prove a\asertion{2}. Let $u\in \huaa{H}$ and $v\in \hua{A}{-}{H}$. By definition, we have
	\begin{eqnarray*}[p,q,v]&=&(pq)v-p(qv)
		\\&=&(pq)v-(qp)v
		\\
		&=&[p,q]v.\end{eqnarray*}
	According to
	identity \eqref{eq:Apq(v,u)=<[p,q,v],u>-[p,q,<v,u>]+pAq(v,u)+Ap(qv,u) } and Lemma \ref{lem:lem57}, we obtain
	\begin{eqnarray*}
		0&=&\fx{[p,q,v]}{u}-[p,q,\fx{v}{u}]\\
		&=&\fx{[p,q]v}{u}-[p,q,\fx{v}{u}]\\
		&=&[p,q]\fx{v}{u}+[[p,q],v,u]-[p,q,\fx{v}{u}]\\
		&=&(pq-qp)\fx{v}{u}-(pq)\fx{v}{u}+p(q\fx{v}{u})\\
		&=&p(q\fx{v}{u})-(qp)\fx{v}{u}.
	\end{eqnarray*}
	This means the  map \begin{eqnarray*}
		f:\O&\to &\O\\
		x&\mapsto& x\fx{v}{u}
	\end{eqnarray*}
	satisfies
	$$pf(q)=f(qp)$$for all $p,q\in \O$. By Lemma \ref{lem:f=0 f(xq)=qf(x)}, we obtain that $f=0$ and hence $\fx{v}{u}=0$.
\end{proof}


\begin{lemma}\label{lem:tensor product}
	Let $H$ be a Hilbert left $\O$-module. If $\hua{A}{-}{H}=0$, then  there exists an isomorphism of Hilbert left $\O$-modules
	\begin{eqnarray}
	H\cong\O\otimes_\R\huaa{H}.
	\end{eqnarray}
	Here $\O$ is a Hilbert left $\O$-module as defined in Example \ref{eg:o,o-}, and $\huaa{H}$ is a real  Hilbert space with the inner product induced from  $H$ by restriction.
  More precisely, for any $p,q\in \O$ and $x,y\in \huaa{H}$, we have
		\begin{eqnarray}
		\fx{px}{qy}=\fx{p}{q}_\O\fx{x}{y}.
		\end{eqnarray}
\end{lemma}

\begin{proof}
	The $\O$-Hilbert space structure of    $\O\otimes_\R\huaa{ H}$ is given as in Definition \ref{def:tensor}. That is,	
	\begin{eqnarray}
	\Big\langle{\sum_{i=0}^7e_i\otimes_\R x_i}, \ \ { \sum_{i=0}^7e_i\otimes_\R
		y_i}\Big\rangle_{\O\otimes_\R\huaa{ H}}=\sum_{i,j=0}^7(e_i\overline{e_j})\fx{x_i}{y_j}
	\end{eqnarray}
	for any
	$x_i,y_i\in \huaa{ H}$.
	
	Since $\hua{A}{-}{H}=0$, it follows from Theorem \ref{thm:M=OA+OA^-} that $$H=\O\huaa{ H}.$$
	We consider the  canonical map
	\begin{eqnarray*}
		\varphi:\quad	\O\otimes_\R\huaa{ H}&\to& H\\
		\sum_{i=0}^7e_i\otimes_\R x_i&\mapsto&\sum_{i=0}^7e_i x_i.
	\end{eqnarray*}
	One can easily check that this is an $\O$-isomorphism of left $\O$-modules.
	It  follows from identity \eqref{eq:<pu,pv>in lemma} and Lemma \ref{cor:ass and conj ass in H} that
	\begin{eqnarray*}
		\Big\langle{\sum_{i=0}^7e_i x_i}, \ \ {\sum_{i=0}^7e_i y_i}\Big\rangle&=&\sum_{i,j=0}^7\fx{e_ix_i}{e_jy_j}\\
		&=&\sum_{i,j=0}^7 \Big( e_i(\fx{x_i}{y_j}\overline{e_j})+[{e_ie_j},x_i,y_j]+\fx{[e_i,e_j,y_j]}{x_i}
		\Big) \\
		&=&\sum_{i,j=0}^7(e_i\overline{e_j})\fx{x_i}{y_j}.
	\end{eqnarray*}
	This means that $$	\Big\langle{\sum_{i=0}^7e_i x_i}, \ \ {\sum_{i=0}^7e_i y_i}\Big\rangle=\Big\langle{\sum_{i=0}^7e_i\otimes_\R x_i}, \ \ {\sum_{i=0}^7e_i\otimes_\R y_i}\Big\rangle_{\O\otimes_\R\huaa{ H}}.$$
	Hence $\varphi$
	is an isomorphism of $\O$-Hilbert spaces. 
\end{proof}


Every Hilbert left $\O$-module can be associated with another Hilbert left $\O$-module structure.
\begin{mydef}
	Let $(H,\fx{\cdot}{\cdot})$ be a Hilbert left $\O$-module.
	We endow $H$ with another left $\O$-multiplication
	\begin{eqnarray*}
		\O\times H&\to& H\\
		(p,x)&\mapsto& p\hat{\cdot}x:=\overline{p}x.
	\end{eqnarray*}
This is also a left $\O$-module and we denote it by $H^-$.
	We define 
	$$\fx{x}{y}_{H^-}:=\fx{y}{x}$$ for any  $x,y\in H$.
\end{mydef}

We now check that $(H^-,\fx{\cdot}{\cdot}_{H^-})$ is indeed a Hilbert left $\O$-module.

\begin{lemma}\label{eq:lemma-duality-261}
	Let $(H,\fx{\cdot}{\cdot})$ be a Hilbert left $\O$-module. Then $(H^-,\fx{\cdot}{\cdot}_{H^-})$ is  a Hilbert left $\O$-module.
	Moreover, it holds
	\begin{eqnarray}
	\huaa{H}&=&\hua{A}{-}{H^-},\label{eq:hua H hua -H}
	\\
	\huaa{H^-}&=&\hua{A}{-}{H}.\label{eq:hua H- hua- H}
	\end{eqnarray}
	The two Hilbert left $\O$-modules have a close relation between their second associators
	\begin{eqnarray}
	[p,x,y]_{H^-}=[p,x,y]+[\fx{x}{y},p]
	\end{eqnarray}
	for all $p\in \O$ and $x,y\in H$.
\end{lemma}

\begin{proof}
	By direct calculations, we have
	\begin{eqnarray*}
		\re \left(\fx{p\hat{\cdot}x}{y}_{H^-}-p\fx{x}{y}_{H^-}\right)&=&\re \left(\fx{y}{\overline{p}x}-p\fx{y}{x}\right)\\
		&=&\re \left(\fx{y}{x}p-[\overline p,y,x]-p\fx{y}{x}\right)\\
		&=&\re \left([\fx{y}{x},p]-[\overline p,y,x]\right)\\
		&=&0.
	\end{eqnarray*}
	This shows the para-linearity. The proof of hermiticity and positivity is trivial.
	
	We come to prove \eqref{eq:hua H hua -H}.
	By definition,
	\begin{eqnarray*}
		\hua{A}{-}{H^-}&=&\{x\mid (pq)\hat{\cdot} x=q\hat{\cdot}(p\hat{\cdot}x), \forall p,q\in \O \}\\
		&=&\{x\mid (\overline{q}\ \overline{p}) x=\overline{q}(\overline{p}x), \forall p,q\in \O \}\\
		&=&\{x\mid [p,q,x]=0, \forall p,q\in \O \}\\
		&=&\huaa{H}.
	\end{eqnarray*}
	This shows  $$\hua{A}{-}{H^-}=\huaa{H}.$$
	It is easy to see that $$(H^-)^-=H.$$ Hence we get $$\huaa{H^-}=\hua{A}{-}{(H^-)^-}=\hua{A}{-}{H}.$$
	
	Finally, we compute the second associators of $H^-$.
	For any $p\in \O$ and $x,y \in H$, we have
	\begin{eqnarray*}
		[p,x,y]_{H^-}&=&\fx{p\hat{\cdot}x}{y}_{H^-}-p\fx{x}{y}_{H^-}\\
		&=&\fx{y}{\overline{p}x}-p\fx{y}{x}\\
		&=&\fx{y}{x}p-[p,y,x]-p\fx{y}{x}\\
		&=&[p,x,y]+[\fx{x}{y},p].
	\end{eqnarray*}
	This proves the lemma.
\end{proof}

By the duality in Lemma \ref{eq:lemma-duality-261}, we   can extend the result  in Lemma \ref{cor:ass and conj ass in H}  to $\hua{A}{-}{H}$.
\begin{lemma}
	Let $H$ be a Hilbert left $\O$-module. Then \begin{eqnarray}
	\fx{x}{y}\in \R
	\end{eqnarray}
	for all $x,y\in \hua{A}{-}{H}$.
\end{lemma}

\begin{proof} It is a direct consequence of   Lemma \ref{eq:lemma-duality-261} with $H$ replaced by $H^-$.
\end{proof}

\begin{lemma}\label{lem:tensor product o-}
	Let $H$ be a Hilbert left $\O$-module. If $\huaa{H}=0$, then  there exists an  $\O$-isomorphism of $\O$-Hilbert spaces
	\begin{eqnarray}
	H\cong\overline{\O}\otimes_\R\hua{A}{-}{H}.
	\end{eqnarray}
	Here $\overline{\O}$ is a Hilbert left $\O$-module as defined in Example \ref{eg:o,o-}, and $\hua{A}{-}{H}$ is a real  Hilbert space with the inner product induced from that of $H$.
	 More precisely, for any $p,q\in \O$ and $x,y\in \hua{A}{-}{H}$, we have
		\begin{eqnarray}
		\fx{px}{qy}=\fx{\overline{p}}{\overline{q}}_{\overline{\O}}\fx{x}{y}.
		\end{eqnarray}
\end{lemma}

\begin{proof}

	We consider the   $\O$-isomorphism
	\begin{eqnarray*}
		\phi:\quad\overline{\O}\otimes_\R\hua{A}{-}{H}&\to& H\\
		\sum_{i=0}^7e_i\otimes_\R x_i&\mapsto&\sum_{i=0}^7\overline{e_i} x_i.
	\end{eqnarray*}
	Note that the left $\O$-module structure of $\overline{\O}\otimes_\R\hua{A}{-}{H}$ is determined by the $\O$-multiplication of $\overline{\O}$ as follows:
	$$p \Big(\sum_{i=0}^7e_i\otimes_\R x_i\Big):=\sum_{i=0}^7(\overline{p}e_i)\otimes_\R x_i.$$
	Hence it follows from $x_i\in \hua{A}{-}{H}$ that \begin{eqnarray*}
		\phi\Big(p \Big(\sum_{i=0}^7e_i\otimes_\R x_i\Big)\Big)&=&\sum_{i=0}^7\overline{(\overline{p}e_i)} x_i\\
		&=&\sum_{i=0}^7(\overline{e_i}p) x_i\\
		&=&p\sum_{i=0}^7\overline{e_i} x_i\\
		&=&p\phi\Big(\sum_{i=0}^7e_i\otimes_\R x_i\Big).
	\end{eqnarray*}
	This proves that $\phi$ is indeed an $\O$-isomorphism.
	
	We next consider the induced module $H^-$.
	By \eqref{eq:hua H hua -H} and \eqref{eq:hua H- hua- H}, we have $$\hua{A}{-}{H^-}=\huaa{H}=0$$ and $$\huaa{H^-}=\hua{A}{-}{H}.$$
	Using Lemma \ref{lem:tensor product}, we obtain
	$$H^-\cong{\O}\otimes_\R\hua{A}{-}{H}.$$
	That is, for any $p,q\in \O$ and $x,y\in \hua{A}{-}{H}$,
	\begin{eqnarray*}
		\fx{p \hat{\cdot}x}{q\hat{\cdot} y}_{H^-}=\fx{p}{q}_{\O}\fx{x}{y}_{H^-},
	\end{eqnarray*}
	or equivalently,
	\begin{eqnarray*}
		\fx{px}{qy}=(\overline{q}p)\fx{x}{y}=\fx{\overline{p}}{\overline{q}}_{\overline{\O}}\fx{x}{y}.
	\end{eqnarray*}
	Therefore,
	\begin{eqnarray*}
		\Big\langle{\phi\Big(\sum_{i=0}^7e_i\otimes_\R x_i\Big)}, \ \ {\phi\Big(\sum_{i=0}^7e_i\otimes_\R y_i\Big)}\Big\rangle&=&
		\Big\langle{\sum_{i=0}^7\overline{e_i} x_i}, \ \ {\sum_{i=0}^7\overline{e_i} y_i}\Big\rangle\\&=&\sum_{i,j=0}^7\fx{\overline{e_i}x_i}{\overline{e_j}y_j}\\
		&=&\sum_{i,j=0}^7(e_j\overline{e_i})\fx{x_i}{y_j}\\
		&=&\sum_{i,j=0}^7\fx{e_i}{e_j}_{\overline{\O}}\fx{x_i}{y_j}\\
		&=&\Big\langle{\sum_{i=0}^7e_i\otimes_\R x_i}, \ \ {\sum_{i=0}^7e_i\otimes_\R y_i}\Big\rangle_{\overline{\O}\otimes_\R\hua{A}{-}{H}}.
	\end{eqnarray*}
	This proves that $\phi$ is an isomorphism of $\O$-Hilbert spaces.
\end{proof}

\begin{thm}[\textbf{Tensor product decomposition}]\label{thm:tensor product}
	Let $H$ be a Hilbert left $\O$-module. Then  there exists an     $\O$-isomorphism of $\O$-Hilbert spaces
	\begin{eqnarray}\label{eq:general hilbert space}
	H\cong  \Big(\O\otimes_\R\huaa{H}\Big)\oplus \Big(\overline{\O}\otimes_\R\hua{A}{-}{H}\Big).
	\end{eqnarray}
	
\end{thm}
\begin{proof}
	By Theorem \ref{thm:M=OA+OA^-}, we have $$H=\O\huaa{H}\oplus \O\hua{A}{-}{H}.$$
	Denote $H_1=\O\huaa{H}$ and $H_2=\O\hua{A}{-}{H}$. Then by Lemmas \ref{lem:tensor product} and   \ref{lem:tensor product o-}, $$H_1\cong\O\otimes_\R\huaa{H},\quad H_2\cong\overline{\O}\otimes_\R\hua{A}{-}{H}.$$
	It remains to show that $H_1$ is orthogonal to $H_2$.
	Recall identity \eqref{eq:<pu,pv>in lemma}:
	$$\fx{pu}{qv}=(p\fx{u}{v})\overline{q}+[{pq},u,v]+\fx{[p,q,v]}{u}.$$
	If $u\in \hua{A}{-}{H}$ and $v\in \huaa{H}$, then it follows from \eqref{eq:<uv>=0} and  Lemma \ref{lem:lem57} that $$\fx{pu}{qv}=0.$$
	This shows that
	$H_1$ is orthogonal to $H_2$ as desired.
\end{proof}
\begin{rem}
 	Note that the right-hand side of equality \eqref{eq:general hilbert space} with $\huaa{H}$ and $\hua{A}{-}{H}$ replaced by any two real Hilbert spaces is always an $\O$-Hilbert space. Hence this theorem gives a complete classification of Hilbert left $\O$-modules.
\end{rem}

	\section{Weak associative orthonormal basis and Parseval's theorem}
 	As observed by Goldstine and Horwitz  in the appendix of \cite{goldstine1964hilbert2},   the  Parseval theorem may fail  for any  orthonormal basis in an $\O$-Hilbert space. We introduce a new notion of weak associative orthonormal basis and  show that the Parseval equality holds for an orthonormal basis if and only if it is  weak associative.
\begin{mydef}
	A subset $S=\{x_{\alpha}\}_{\alpha\in \Lambda}$ of an $\O$-Hilbert space is said to be an \textbf{orthonormal system} if $$\fx{x_{\alpha}}{x_{\beta}}=\delta_{\alpha \beta}.$$
	An orthonormal system $S=\{x_{\alpha}\}_{\alpha\in \Lambda}$ is said to be an \textbf{orthonormal basis} if there does not exist other orthonormal system $S'$ such that $S\subsetneqq S'$.
\end{mydef}
{As usual, we have the following characterizations of an orthonormal basis. The proof runs in the same manner as in the classical case and is omitted here.}
\begin{lemma}\label{lem:ONB<->nexist nonzero elem perp to OS}
	Let $S$ be an orthonormal system. Then $S$ is an orthonormal   basis if and only if  there does not exist non-zero element which is orthogonal to every element of $S$.
	
\end{lemma}

In contrast to the classical cases, it is meaningless to
  define the dimension of an 	$\O$-Hilbert space as the cardinality  of an orthonormal basis as shown in the following
  special  case  of $\O^2$.
\begin{eg}\label{eg:onb of O2}
	The $\O$-Hilbert space $\O^2$ admits two sets of basis.  One has cardinality $2$ and the other has cardinality $4$.
	
	(1). The $\O$-Hilbert space $\O^2$ admits an  orthonormal basis given by
	$$(0,1),\quad(1,0).$$

	(2).  The $\O$-Hilbert space $\O^2$ admits an alternative orthonormal basis given by
	
	\begin{eqnarray}\label{eq:ex in O2}
	x_1=\dfrac{1}{\sqrt{2}}(e_1,e_2),\quad
	x_2=\dfrac{1}{\sqrt{2}}(e_4,e_7),\quad
	x_3=\dfrac{1}{\sqrt{2}}(e_6,e_5),\quad
	x_4=\dfrac{1}{\sqrt{2}}(1,e_3).
	\end{eqnarray}

	
	To show  that   \eqref{eq:ex in O2} is an orthonormal basis, we first observe that
	  $\{x_{n}\}_{n=1}^4$ is an orthonormal system of $\O^2$. If  there exists an element $x=(a,b)\in \O^2$
	orthogonal to $x_{1},\dots, x_4$, then  from the assumption that $\fx{x}{x_4}=0$ we have $$a=be_3.$$
	Since $\fx{x}{x_1}=0$, we get
	\begin{eqnarray}\label{eqpf:eg b}
	(be_3)e_1+be_2=2be_2+[b,e_3,e_1]=0.
	\end{eqnarray}
	Using identity \eqref{eq:<[qpx]y>0=<[qpy]x>} in the $\O$-Hilbert space $\O$, we obtain
	\begin{eqnarray}\label{eqpf:[IJx]}
	[I,J,x]\in {\mathbb{H}_{I,J}}^{\perp_\R}
	\end{eqnarray}
	for any mutually orthogonal pair $(I,J)$ of $\mathbb{S}_6$ and all $x\in \O$. Here the notation $\mathbb{H}_{I,J}$ stands for the subset spanned by
	$\{1,I,J,IJ\}$, which was  introduced in Subsection \ref{subsec:O}.
	Combining \eqref{eqpf:[IJx]} with \eqref{eqpf:eg b}, we get
	$$be_2\in {\mathbb{H}_{e_1,e_3}}^{\perp_\R}.$$
	This yields that $b_0=b_1=b_2=b_3=0$  if we write $$b=b_0+\sum_{i=1}^7e_ib_i.$$
	The assumptions that  $\fx{x}{x_2}=\fx{x}{x_3}=0$ imply that
	\begin{numcases}{}
	2be_7+[b,e_3,e_4]=0,\notag\\
	2be_5+[b,e_3,e_6]=0.\notag
	\end{numcases}
	By similar argument, we conclude that $b_4=b_5=b_6=b_7=0$ and hence $b=0$. This  means that $x=0$. We thus conclude from Lemma \ref{lem:ONB<->nexist nonzero elem perp to OS}  that $\{x_{n}\}_{n=1}^4$ is an orthonormal basis.

\end{eg}

The Parseval equality may not hold for
the orthonormal basis \eqref{eq:ex in O2}.
		Indeed, we take $y=(1,0)\in \O^2$. By direct calculation, we have
	$$\fx{y}{x_1}=\dfrac{1}{\sqrt{2}}\overline{e_1},\quad \fx{y}{x_2}=\dfrac{1}{\sqrt{2}}\overline{e_4},\quad
	\fx{y}{x_3}=\dfrac{1}{\sqrt{2}}\overline{e_6},\quad
	\fx{y}{x_4}=\dfrac{1}{\sqrt{2}}.$$
	It follows that $$y-\sum_{n=1}^4\fx{y}{x_n}x_n=(-1,e_3)$$
	and $$\sum_{n=1}^4\fsh{\fx{y}{x_n}}^2\neq \fsh{y}^2.$$


	  Denote $p_i=\fx{y}{x_i}$ for $i=1,\dots,4$,
	 Using identity \eqref{eq:<[qpx]y>0=<[qpy]x>} and the left alternativity of associators, we obtain
	\begin{align*}
	\sum_{m,n=1}^4\fx{[\fx{y}{x_n},\fx{y}{x_m},x_m]}{x_n}_\spr&=2\sum_{m<n}\fx{[\fx{y}{x_n},\fx{y}{x_m},x_m]}{x_n}_\spr\\
	&=2\big(\fx{[p_1,p_2,x_2]}{x_1}_\spr+\fx{[p_1,p_3,x_3]}{x_1}_\spr+\fx{[p_2,p_3,x_3]}{x_2}_\spr\big)\\
	&=\dfrac{1}{2} \re([{e_1},{e_4},e_7]\overline{e_2}+[{e_1},e_6,{e_5}]\overline{e_2}+[{e_4},{e_6},e_5]\overline{e_7})\\
	&=3.
	\end{align*}
	Therefore, we have
	$$\sum_{n=1}^4\left|\fx{y}{x_n}\right|^2+\Big\|{y-\sum_{n=1}^4\fx{y}
{x_n}x_n}\Big\|^2-\sum_{m,n=1}^2\fx{[\fx{y}{x_n},\fx{y}{x_m},x_m]}{x_n}_\spr=\fsh{y}^2.$$
In fact, the above  equality even  holds  generally.

\begin{lemma}\label{lem:Bessel}
	Let $\{x_{n}\}_{n=1}^N$ be an orthonormal system  of an $\O$-Hilbert space $H$. Then for all $x\in H$ we have
	\begin{eqnarray}\label{eq:bessel}
	\left|\left|x\right|\right|^2=\sum_{n=1}^N\left|\fx{x}{x_n}\right|^2+
\Big\|{x-\sum_{n=1}^N\fx{x}{x_n}x_n}\Big\|^2-\sum_{m,n=1}^N\fx{[\fx{x}{x_n},\fx{x}{x_m},x_m]}{x_n}_{\spr}.
	\end{eqnarray}
	In particular, there holds the octonionic version of Bessel's inequality:
	\begin{eqnarray}
	\sum_{n=1}^N\left|\fx{x}{x_n}\right|^2-\sum_{m,n=1}^N\fx{[\fx{x}{x_n},\fx{x}{x_m},x_m]}{x_n}_{\spr}\leqslant \fsh{x}^2.
	\end{eqnarray}
	
\end{lemma}


\begin{proof}
	Set $$p_n=\fx{x}{x_n}, \quad u=\sum_{n=1}^N p_nx_n, \quad  v=x-u.$$
	By calculation, we  have
	\begin{align}\label{eqpf:fsh x}
	\fsh{x}^2&=\fx{u+v}{u+v}\notag\\
	&=\fsh{u}^2+\fsh{v}^2+2\re\fx{u}{x-u}\notag\\
	&=\fsh{v}^2-\fsh{u}^2+2\re\fx{u}{x}.
	\end{align}
	Since
	\begin{align*}
	\fx{u}{x}&=\Big\langle{\sum_{i=1}^N p_nx_n}, \ \ {x}\Big\rangle\\
	&=\sum_{n=1}^N p_n\fx{x_n}{x}+[{p_n},x_n,x]\\
	&=\sum_{n=1}^N\left|\fx{x}{x_n}\right|^2+\sum_{n=1}^N [{p_n},x_n,x],
	\end{align*}
	we conclude that $$\re\fx{u}{x}=\sum_{n=1}^N\left|\fx{x}{x_n}\right|^2.$$
	By \eqref{eq:<pu,pv>in lemma},
	we obtain
	\begin{eqnarray*}
	\fsh{u}^2&=&\Big\langle{\sum_{n=1}^N p_nx_n}, \ \ {\sum_{m=1}^N p_mx_m}\Big\rangle
\\
	&=&\sum _{m,n=1}^N\big(p_n\fx{x_n}
{x_m}\big)\overline{p_m}+[{p_np_m},x_n,x_m]+\fx{[p_n,p_m,x_m]}{x_n}
\\	&=&\sum_{n=1}^N\left|\fx{x}{x_n}\right|^2+\sum_{m,n=1}^N[{p_np_m},x_n,x_m]+\fx{[\fx{x}{x_n},\fx{x}{x_m},x_m]}{x_n},
	\end{eqnarray*}
	so that
	\begin{align*}
	\fsh{u}^2 =\re \fsh{u}^2 &=\sum_{n=1}^N\left|\fx{x}{x_n}\right|^2+\sum_{m,n=1}^N\fx{[\fx{x}{x_n},\fx{x}{x_m},x_m]}{x_n}_\spr.
	\end{align*}
	Hence \eqref{eqpf:fsh x} becomes
	\begin{eqnarray*}	\fsh{x}^2&=&\fsh{v}^2-\Big(\sum_{n=1}^N\left|\fx{x}{x_n}\right|^2+\sum_{m,n=1}^N\fx{[\fx{x}{x_n},\fx{x}{x_m},x_m]}{x_n}_\spr\Big)+2\sum_{n=1}^N\left|\fx{x}{x_n}\right|^2\\ &=&\sum_{n=1}^N\left|\fx{x}{x_n}\right|^2+\Big\|{x-\sum_{n=1}^N\fx{x}{x_n}x_n}\Big\|^2-\sum_{m,n=1}^N\fx{[\fx{x}{x_n},\fx{x}{x_m},x_m]}{x_n}_\spr.
	\end{eqnarray*}
	This proves the lemma.
\end{proof}
\begin{rem}
	In Example \ref{eg:onb of O2},
	it is worth pointing out that $\{x_{n}\}_{n=1}^4$ is not $\O$-linearly independent. Indeed, if we denote $$x_5=\dfrac{1}{\sqrt{2}}(1,-e_3),$$ then we have
	$$x_1=e_1x_5,\quad x_2=e_4x_5,\quad x_3=e_6x_5.$$
	Moreover, it is easy to check that $\{x_{n}\}_{n=4}^5$  is also an orthonormal basis.
	By direct calculation, for $y=(1,0)$ we have
	$$y-\sum_{i=4}^5\fx{y}{x_i}x_i=0$$
	and $$\fsh{y}^2=\sum_{i=4}^5\left|\fx{y}{x_i}\right|^2.$$
	In fact, {this is not accidental in terms of   
		the corollary	 below. The point is that the orthonormal basis $\{x_{n}\}_{n=4}^5$ satisfies
		$$[p,x_4,x_5]=0$$ for all $p\in \O$.}
\end{rem}	
\begin{cor}\label{cor:bessel ineq}
	Under the assumption of Lemma \ref{lem:Bessel} and suppose $\{x_{n}\}_{n=1}^N$ satisfies
	\begin{eqnarray}\label{eq:sons condition}
	[p,x_n,x_m]=0
	\end{eqnarray}
	for all  $p\in\spo$ and for each $ m,n=1,\dots,N$.
	Then we have
	\begin{eqnarray}\label{eq:beseel eq}
	\left|\left|x\right|\right|^2=\sum_{n=1}^N\left|\fx{x}{x_n}\right|^2
+\Big\|{x-\sum_{n=1}^N\fx{x}{x_n}x_n}\Big\|^2	
	\end{eqnarray}	
	and in this case we obtain the classical Bessel's inequality
	\begin{eqnarray}\label{eqcor:bessel ineq}
	\sum_{n=1}^N\left|\fx{x}{x_n}\right|^2\leqslant \fsh{x}^2.
	\end{eqnarray}
\end{cor}
\begin{proof}
	If $	[p,x_n,x_m]=0$, then by taking the real part on both sides of \eqref{eq:Apq(v,u)=<[p,q,v],u>-[p,q,<v,u>]+pAq(v,u)+Ap(qv,u) }, we have
	$$\sum_{m,n=1}^N\fx{[\fx{x}{x_n},\fx{x}{x_m},x_m]}{x_n}_{\spr}=-\sum_{m,n=1}^N\re\left(\fx{x}{x_n}	[\fx{x}{x_m},x_n,x_m]\right)=0. $$
	Therefore,  \eqref{eq:beseel eq} follows from \eqref{eq:bessel} directly.
	This completes the proof.
\end{proof}
This inspires us to introduce the following notion.
\begin{mydef}\label{def:weak ass}
	An orthonormal basis $S=\{x_{\alpha}\}_{\alpha\in \Lambda}$ is said to be a \textbf{weak associative orthonormal basis} if $S$  satisfies  condition \eqref{eq:sons condition} in Corollary \ref{cor:bessel ineq}.
	
\end{mydef}

\begin{lemma}\label{lem:CONB is ONB R}
	 If $S=\{x_{\alpha}\}_{\alpha\in \Lambda}$ is a weak associative orthonormal basis in  a Hilbert left $\O$-module   $H$,
	 	then $$\widetilde{S}:=\{e_ix_\alpha\mid i=0,\dots,7, \ \alpha\in \Lambda\}$$  is a    real orthonormal basis of the real Hilbert space $(H,\fx{\cdot}{\cdot}_{\spr})$.
\end{lemma}

\begin{proof}
	According to identities \eqref{eq:<pu,pv>in lemma} and  \eqref{eq:Apq(v,u)=<[p,q,v],u>-[p,q,<v,u>]+pAq(v,u)+Ap(qv,u) }, we have
	\begin{align*}
	\fx{e_ix_\alpha}{e_jx_\beta}_{\spr}&=\re \left(e_i\fx{x_\alpha}{x_\beta}\right)\overline{e_j}+[e_i{e_j},x_\alpha,x_\beta]+\fx{[e_i,e_j,x_\beta]}{x_\alpha}_\R\\
	&=\re(e_i\overline{e_j}\delta_{\alpha\beta})-\re(e_i[e_j,x_\beta,x_\alpha])\\
	&=\delta_{ij}\delta_{\alpha\beta}
	\end{align*}
	for each $i,j=0,\dots,7$ and each $\alpha,\beta\in \Lambda$.
	Hence $\widetilde{S}$ is a real orthonormal system of  $(H,\fx{\cdot}{\cdot}_{\spr})$.
	We next show that $\widetilde{S}$ is a real orthonormal basis.
	If not, then there exists a non-zero element $x\in H$ such that $x$ is orthogonal to $\widetilde{S}$ with respect to the real inner product $\fx{\cdot}{\cdot}_{\spr}$.
	Thus we have $$\fx{x_\alpha}{x}=\left\langle x_\alpha,x\right\rangle _{\spr}-\sum_{i=0}^7 \left\langle e_ix_\alpha,x\right\rangle _{\spr}e_i=0.$$
	In view of Lemma \ref{lem:ONB<->nexist nonzero elem perp to OS}, this is impossible.
\end{proof}

Now we are in   position	to establish the Parseval theorem for  weak associative orthonormal   bases.
\begin{thm}[\textbf{Parseval theorem}]\label{thm:Parseval thm}
	Let $H$ be an $\O$-Hilbert space and $S=\{x_{\alpha}\}_{\alpha\in \Lambda}$ be a weak associative orthonormal   basis. Then  any $x\in H$  can be  uniquely expressed   as
	\begin{equation*}
	x=\sum_{\alpha\in \Lambda}\fx{x}{x_{\alpha}}x_{\alpha}
	\end{equation*}
	and there holds
	\begin{equation*}
	\fsh{x}^2=\sum_{\alpha\in \Lambda}|\fx{x}{x_{\alpha}}|^2.
	\end{equation*}
\end{thm}

\begin{proof}
	In view of Cauchy-Schwarz inequality,  we get  $$\left|\fx{x}{x_\alpha}\right|\leqslant\fsh{x}$$
	for each $\alpha\in \Lambda$.
	Note that
	$$ \bigcup_{k=1}^\infty\big[\dfrac{1}{k+1}\fsh{x},\dfrac{1}{k}\fsh{x}\big]=(\; 0,\ \fsh{x}\; ].$$
	Thanks to Corollary \ref{cor:bessel ineq}, there are only countable $\alpha$, say $(\alpha_j)_{j=1}^\infty$, such that $\fx{x}{x_{\alpha_j}}\neq 0$. 
	Then  by Bessel's inequality \eqref{eqcor:bessel ineq} in Corollary \ref{cor:bessel ineq}, we have
	\begin{eqnarray}\label{eqpf:sum<inft}
	\sum_{j=1}^{\infty}\left|\fx{x}{x_{\alpha_j}}\right|^2< +\infty.
	\end{eqnarray}
	Set $$y_n=\sum_{j=1}^n\fx{x}{x_{\alpha_j}}x_{\alpha_j}.$$ For any $m<n$,  we have $$y_n-y_m=\sum_{j=m+1}^n\fx{x}{x_{\alpha_j}}x_{\alpha_j}.$$
	Since $S$ is a weak associative orthonormal basis, we get
	\begin{eqnarray*}
		\fx{y_n-y_m}{x_{\alpha_k}}&=&\sum_{j=m+1}^n\fx{\fx{x}{x_{\alpha_j}}x_{\alpha_j}}{x_{\alpha_k}}\\
		&=&\sum_{j=m+1}^n\fx{x}{x_{\alpha_j}}\fx{x_{\alpha_j}}{x_{\alpha_k}}+[{\fx{x}{x_{\alpha_j}}},x_{\alpha_j},x_{\alpha_k}]\\
		&=&\fx{x}{x_{\alpha_k}}.
	\end{eqnarray*}	
	It then follows from Corollary \ref{cor:bessel ineq} again that
	\begin{eqnarray*}		\fsh{y_n-y_m}^2&=&\sum_{k=m+1}^n\abs{\fx{y_n-y_m}{x_{\alpha_k}}}^2+\Big\|{(y_n-y_m)-\sum_{k=m+1}^n\fx{y_n-y_m}{x_{\alpha_k}} x_{\alpha_k}  }\Big\|^2\\		&=&\sum_{k=m+1}^n\abs{\fx{x}{x_{\alpha_k}}}^2+\Big\|{\sum_{j=m+1}^n\fx{x}{x_{\alpha_j}}x_{\alpha_j}-\sum_{k=m+1}^n\fx{x}{x_{\alpha_k}} x_{\alpha_k}  }\Big\|^2\\
		&=&\sum_{k=m+1}^n\abs{\fx{x}{x_{\alpha_k}}}^2.
	\end{eqnarray*}
	Thus we conclude from \eqref{eqpf:sum<inft} that $\{y_n\}_{n=1}^{\infty}$ is a Cauchy sequence. From the completeness of $H$, there exists an element $x'\in H$ such that
	$$\lim_{n\rightarrow+\infty}y_n= x'.$$
	
	We next prove that $x'=x$.
	In view of Lemma \ref{lem:CONB is ONB R}, it suffices to prove
	\begin{eqnarray}
	\fx{x-x'}{e_ix_{\alpha}}_{\spr}&=0
	\end{eqnarray}
	for each $\alpha\in \Lambda$ and $i=1,\dots,7$.
	For any $x_{\alpha_k}$,
	we get
	\begin{eqnarray*}
		\fx{x-x'}{e_ix_{\alpha_k}}_\R
		&=&\re\big(\fx{x-x'}{x_{\alpha_k}}\overline{e_i} +[{e_i},x-x',x_{\alpha_k}]\big)\\
		&=&\re\Big(\fx{x}{x_{\alpha_k}}\overline{e_i}
-\lim_{n\rightarrow+\infty}\sum_{j=1}^n\fx{\fx{x}{x_{\alpha_j}}
x_{\alpha_j}}{{x_{\alpha_k}}}\overline{e_i}\Big)\\
		&=&\re\Big(\fx{x}{x_{\alpha_k}}\overline{e_i}-\lim_{n\rightarrow+\infty}\sum_{j=1}^n\big(\fx{x}{x_{\alpha_j}}\fx{x_{\alpha_j}}{x_{\alpha_k}} \big)\overline{e_i}\Big) \\
		&=&0.
	\end{eqnarray*}
	For any $x_{\alpha}$ with $\alpha\neq \alpha_j,\ j=1,2,\dots,$  we also  have
	\begin{eqnarray*}
		\fx{x-x'}{e_ix_{\alpha}}_\R
		&=&\re\big(\fx{x-x'}{x_{\alpha}}\overline{e_i} +[{e_i},x-x',x_{\alpha}]\big)\\
		&=&\re\Big(\fx{x}{x_{\alpha}}\overline{e_i}
-\lim_{n\rightarrow+\infty}\sum_{j=1}^n\fx{\fx{x}{x_{\alpha_j}}
x_{\alpha_j}}{{x_{\alpha}}}\overline{e_i}\Big)\\
		&=&\re\Big(-\lim_{n\rightarrow+\infty}\sum_{j=1}^n\big(\fx{x}{x_{\alpha_j}}\fx{x_{\alpha_j}}{x_{\alpha}} \big)\overline{e_i}\Big) \\
		&=&0.
	\end{eqnarray*}
	%
	This proves  $x=x'$. Hence we have
	$$x=\sum_{j=1}^{\infty}\fx{x}{x_{\alpha_j}}x_{\alpha_j}=\sum_{\alpha\in \Lambda}\fx{x}{x_{\alpha}}x_{\alpha}.$$
	By Corollary \ref{cor:bessel ineq} we obtain	 $$\fsh{x}^2=\sum_{j=1}^n\Big|\fx{x}{x_{\alpha_j}}\Big|^2+
\Big\|{x-\sum_{j=1}^n\fx{x}{x_{\alpha_j}}x_{\alpha_j}}\Big\|^2$$
for any $n\in \mathbb{N}$.
	Letting $n\rightarrow \infty$, we get
	$$\fsh{x}^2=\sum_{j=1}^{\infty}\left|\fx{x}{x_{\alpha_j}}\right|^2=\sum_{\alpha\in \Lambda}|\fx{x}{x_{\alpha}}|^2.$$ The uniqueness of the decomposition follows from the condition  \eqref{eq:sons condition}. This completes the proof.
\end{proof}

\begin{rem}
	For any $\O$-Hilbert space, there always  exists a weak associative orthonormal   basis.
	This is because   there always exists an orthonormal   basis in a real Hilbert space.
	Note that by Theorem \ref{thm:tensor product}, an $\O$-Hilbert space $H$ can be decomposed into an orthogonal sum. Let $\{x_{\alpha}\}_{\alpha\in \Lambda}$ be an orthonormal   basis of the real Hilbert space $(\huaa{H},\fx{\cdot}{\cdot}_\R)$ and $\{x_{\beta}\}_{\beta\in \Lambda'}$ be an orthonormal   basis of the real Hilbert space $(\hua{A}{-}{H},\fx{\cdot}{\cdot}_\R)$. Then $$S:=\{x_\gamma\}_{\gamma\in \Lambda \cup \Lambda'}$$ is a weak associative orthonormal   basis.

\end{rem}

At last, we show that every orthonormal   basis for which the Parseval theorem holds is also weak associative.

\begin{mydef}
	{	An orthonormal   basis $S=\{x_{\alpha}\}_{\alpha\in \Lambda}$ is  said to be a Hilbert basis, if for
		any $x\in H$  it holds
		\begin{equation*}
		\fsh{x}^2=\sum_{\alpha\in \Lambda}|\fx{x}{x_{\alpha}}|^2.
		\end{equation*}
	}
	
\end{mydef}

\begin{thm}
	Let $S$ be an orthonormal   basis. Then $S$ is a Hilbert basis if and only if $S$ is weak associative.
	
\end{thm}
\begin{proof}
	In view of Theorem \ref{thm:Parseval thm}, it suffices to show that every  Hilbert basis  is weak associative.
	Let  $S=\{x_{\alpha}\}_{\alpha\in \Lambda}$ be a Hilbert basis. Then for  $x=px_\beta$, we have
	\begin{align*}
	\fsh{x}^2&=\sum_{\alpha\in \Lambda}\abs{\fx{px_\beta}{x_{\alpha}}}^2\\
	&=\sum_{\alpha\in \Lambda}\abs{(p\delta_{\alpha\beta}+[{p},x_\beta,x_\alpha])}^2\\
	&=\abs{p}^2+\sum_{\alpha\neq \beta}\abs{[{p},x_\beta,x_\alpha]}^2.
	\end{align*}
	It follows from $\fsh{x}=\abs{p}$ that $$\sum_{\alpha\neq \beta}\abs{[{p},x_\beta,x_\alpha]}^2=0.$$
	This implies that  $$[{p},x_\beta,x_\alpha]=0$$
	for all $x_\beta,x_\alpha\in S$. This proves that $S$ is weak associative.
\end{proof}

\section{Cayley-Dickson algebras as  $\O$-Hilbert spaces}

Any Cayley–Dickson algebra can be endowed with   a  weak $\O$-bimodule structure and an $\O$-inner product.

At first, we introduce the notion of  weak $\O$-bimodules.
	Let $M$ be a left $\O$-module. There  associates an involution
$$C: M\to M,$$
called \textbf{conjugate} map.
Since
 any $x\in M$ can be written as
	$$x=\sum_{i=0}^7e_ix_i+\sum_{i=0}^7e_ix^-_i,$$
	where $x_i\in \huaa{M}$ and $x^-_i\in \hua{A}{-}{M}$,  we define the conjugate of $x$ as
$$C(x)=\overline{x}:=\sum_{i=0}^7\overline{e_i}x_i-\sum_{i=0}^7e_ix^-_i.$$
Then by definition, we have
\begin{eqnarray*}\overline x
	=
	\begin{cases} x, \qquad & x\in \huaa{M},
		\\ \\
		-x, \qquad & x\in \hua{A}{-}{M}.
	\end{cases}
\end{eqnarray*}
This means   \begin{eqnarray}\label{eq:ChuaM=huaM}
C(\huaa{M})=\huaa{M},\quad C(\hua{A}{-}{M})=\hua{A}{-}{M}.
\end{eqnarray}

\begin{mydef}	We call $M$ a \textbf{weak $\O$-bimodule} if $M$ is both  a left $\O$-module and  a right $\O$-module satisfying the compatibility conditions
	\begin{eqnarray}\label{eq:-px=p-x-}
	\overline{px}=\overline{x}\ \overline{p}
	\end{eqnarray}
	for all $x\in M$ and $p\in \O$.
\end{mydef}
\begin{rem}
	The notion of a bimodule   for a class of algebras has been introduced by Eilenberg \cite{eilenberg1948extensions}. Following \cite{jacobson1954structure,Schafer1952repaltalg}, we call $M$
an $\O$-bimodule if
 $M$ is both  a left $\O$-module and  a right $\O$-module satisfying the compatibility conditions
	$$[p,q,x]=[x,p,q]=[q,x,p]$$ for all $p,q\in \O$ and $x\in M$.
	The regular bimodule $\O$ is known to be the only
	irreducible $\O$-bimodule \cite{jacobson1954structure,Schafer1952repaltalg}.
	 We remark that a weak $\O$-bimodule may not be   an $\O$-bimodule in general.
	
\end{rem}

\begin{rem}
		Let  $M$ be a weak $\O$-bimodule.
	\begin{enumerate}
		\item  	For all $p\in \O$ and $x\in M$, we have
	\begin{eqnarray}\label{eq:xp=p-x-}
	\overline{xp}=\overline{p}\ \overline{x}.
		\end{eqnarray}
Indeed,  it follows from \eqref{eq:-px=p-x-} that
		$$px=\overline{\overline{px}}=\overline{\overline{x}\ \overline{p}}$$ which becomes
		 (\ref{eq:xp=p-x-}) after replacing $p, x$ by $\overline p, \overline x$, respectively.
	
\item 		The right multiplication of $M$ is determined by its left multiplication
due to  \eqref{eq:xp=p-x-}.
Conversely, every left $\O$-module can be viewed as a weak $\O$-bimodule by endowed with such right multiplication. That is, a weak $\O$-bimodule is nothing but a left $\O$-module endowed with a  certain right multiplication.
\item  We denote the left nucleus by $$\mathscr{A}_{L} (M):=\{x\in M \mid [p,q,x]=0, \forall p,q\in \O\}$$ and the right nucleus by  $$\mathscr{A}_{R}(M):=\{x\in M \mid [x,p,q]=0, \forall p,q\in \O\}.$$
Similar notations $\mathscr{A}^-_{L}(M),\ \mathscr{A}^-_{R}(M)$ can be also defined.
One can check that $$\mathscr{A}_{R}(M)=C(\mathscr{A}_{L}(M)),\quad \mathscr{A}^-_{R}(M)=C(\mathscr{A}^-_{L}(M)).$$
Combining this with \eqref{eq:ChuaM=huaM}, we have
\begin{eqnarray}\label{eq:huaar}
\mathscr{A}_{R}(M)=\mathscr{A}_{L}(M),\quad \mathscr{A}^-_{R}(M)=\mathscr{A}^-_{L}(M).
\end{eqnarray}
Henceforth we shall denote
\begin{eqnarray*}
\mathscr{A}(M):=\mathscr{A}_{R}(M)=\mathscr{A}_{L}(M),\quad \mathscr{A}^-(M):=\mathscr{A}^-_{R}(M)=\mathscr{A}^-_{L}(M).
\end{eqnarray*}
\end{enumerate}

\end{rem}
\bigskip

We next consider the algebra of sedenions. The algebra of sedenions $\mathbb{S}$ is constructed from  the algebra of octonions through    the Cayley–Dickson construction \cite  {Dickson1919cdalg}.
The elements in the algebra of sedenions  $\mathbb{S}$ take  the form
$$s =\sum_{i=0}^{15}
x_ie_i,$$
where $x_i$ are reals, $e_0 = 1$ and $e_1, e_2, \dots, e_{15}$ are imaginary units, i.e.,  $$e_i^2=-1$$ for all $i=1,\dots,15$.
The conjugate of $s$ is defined as
$$\overline{s}=x_0-\sum_{i=1}^{15}
x_ie_i.$$
We regard   $\O$   as the subalgebra  of  $\mathbb{S}$
generated by $e_0,e_1, \dots , e_7$. By the  Cayley-Dickson construction,  any $s\in \mathbb{S}$ can be expressed uniquely as $$s=a+be_8$$ with $a,b\in \O$ so that
 the conjugate of $s$ becomes
$$\overline{s}=\overline{a}-be_8$$ and
the multiplication   in $\mathbb{S}$ is defined as
\begin{eqnarray}\label{eq:cd prod}
(a+be_8)(c+de_8)=(ac-\overline{d}b)+(da+b\overline{c})e_8.
\end{eqnarray}

We want to show that $\mathbb{S}$ is a weak $\O$-bimodule.
At first, we show that $\mathbb{S}$ is a left $\O$-module with multiplication
\begin{eqnarray*}
\O\times \mathbb{S}&\to& \mathbb{S}\\
(p,x)&\mapsto&px.
\end{eqnarray*}
Indeed, it follows from \eqref{eq:cd prod} that
\begin{eqnarray}\label{eq:e8}
a(de_8)=(da)e_8
\end{eqnarray}
for all $a,d\in \O$.
Hence for all $p\in \O$ and $x=a+
be_8\in \mathbb{S}$, we have
$$p(px)=p(pa+(bp)e_8)=p^2a+(bp^2)e_8=p^2x.$$
This shows that $\mathbb{S}$ is a left $\O$-module.

	One can easily check  that the conjugate of $s$ in sedenions coincides with the  conjugate of $s$ when $\mathbb{S}$  is regarded   as a left $\O$-module. Hence the right multiplication induced by
\eqref{eq:-px=p-x-}
 is
	\begin{eqnarray*}
		\mathbb{S}\times\O &\to& \mathbb{S}\\
		(x,p)&\mapsto&\overline{\overline{p}\ \overline{x}}=xp.
	\end{eqnarray*}

 It can be seen from \eqref{eq:e8} that $$e_8\in \hua{A}{-}{\mathbb{S}}.$$ Therefore we get $$\mathbb{S}\cong \O\oplus \overline{\O}$$ as left   $\O$-modules.

We now define $$\fx{a+be_8}{c+de_8}_{\mathbb{S}}:=a\overline{c}+\overline{d}b.$$In other words, if we denote by $$\pi_{\O}:\mathbb{S}\to \O$$ the projection from $\mathbb{S}$ to $\O$, then for any $s,w\in \mathbb{S}$,
$$\fx{s}{w}_{\mathbb{S}}:=\pi_{\O}(s\overline{w}).$$
One can check that this is an $\O$-inner product.

We move on our discussion to arbitrary
 Cayley–Dickson algebra $\mathbb{A}_n$.    For $n=1,2,3$, $$ \mathbb{A}_1=\spc, \quad \mathbb{A}_2=\mathbb{H}, \quad \mathbb{A}_3=\O.$$
  We shall apply the induction  to endow  $\mathbb{A}_n$ with a  Hilbert $\O$-module structure.

The  Cayley-Dickson algebra  $\mathbb{A}_n$ is defined by induction.
An element $s$ in $\mathbb{A}_n$    can be written as $$s=a+be_{2^{n-1}},$$ where $a,b\in \mathbb{A}_{n-1}$.
 Its
conjugate   is defined as
$$\overline{a+be_{2^{n-1}}}=\overline{a}-be_{2^{n-1}}.$$
 the multiplication of $\mathbb{A}_n$  is defined as
\begin{eqnarray}\label{eq:cdn prod}
(a+be_{2^{n-1}})(c+de_{2^{n-1}})=(ac-\overline{d}b)+(da+b\overline{c})e_{2^{n-1}},
\end{eqnarray}
	where $a,b,c,d\in \mathbb{A}_{n-1}$.

As a real  vector space,
 $\mathbb{A}_n$ admits  a canonical basis
$$e_0,e_1,\dots,e_{2^n-1},$$
where
\begin{eqnarray}\label{eq:e2n+j}
e_{2^{n-1}+j}:=e_je_{2^{n-1}}
\end{eqnarray} for any $j=1,\dots, 2^{n-1}-1$.
The octonions $\O$ can be  viewed as a subset of  $\mathbb{A}_n$ when $n\geqslant 4$ generated    by $e_0,e_1, \cdots, e_7$.

The left $\O$-module structure on $\mathbb{A}_n$ is induced by   the multiplication of $\mathbb{A}_n$:\begin{eqnarray*}
	\O\times \mathbb{A}_n&\to& \mathbb{A}_n\\
	(p,x)&\mapsto&px.
\end{eqnarray*}
It turns out that $\mathbb{A}_n$ is a  weak $\O$-bimodule for $n\geqslant4$.
\begin{lemma}
If $n\geqslant 4$,  then	$\mathbb{A}_n$ is a weak $\O$-bimodule and the two    conjugates coincide when  $\mathbb{A}_n$  is regarded as either a    weak $\O$-bimodule or a  Cayley–Dickson algebra. Moreover, we have an isomorphism of left $\O$-modules  $$\mathbb{A}_n\cong \O^{2^{n-4}}\oplus \overline{\O}^{2^{n-4}}.$$
\end{lemma}
\begin{proof} When $n=4$, we have $\mathbb{A}_4=\mathbb S$ so the result  has been proved. Suppose it holds for $n$, we come to prove the case of  $n+1$.
	
	It follows from \eqref{eq:cdn prod} that $$p(qe_{2^n})=(qp)e_{2^n}$$ for all $p,q\in \O$ and $n\geqslant 4$.
	
		Now we check that the subset $\mathbb{A}_n e_{2n}$ of $\mathbb{A}_{n+1}$ is also a left $\O$-module as well as a right $\O$-module.
	For any $p\in \O$ and $x\in \mathbb{A}_n$,
	\begin{eqnarray}\label{eqpf:ppxe}
	p(p(xe_{2^n}))=p((xp)e_{2^n})=((xp)p)e_{2^n}.
	\end{eqnarray}
	 By induction,  $\mathbb{A}_n$ is a right $\O$-module so that $(xp)p=xp^2$. Therefore \eqref{eqpf:ppxe} becomes
	 $$	p(p(xe_{2^n}))=(xp^2)e_{2^n}=p^2(xe_{2^n}).$$
	 This shows that  $\mathbb{A}_ne_{2^n}$ is a  left $\O$-module. It can be proved that $\mathbb{A}_ne_{2^n}$ is a  right $\O$-module similarly.

	
	 Moreover, it is easy to check that  \begin{eqnarray}\label{eqpf:cd2}
	 \mathscr{A}_{L}{(\mathbb{A}_ne_{2^n})}=\mathscr{A}^-_{R}({\mathbb{A}_n})e_{2^n},\quad \mathscr{A}^-_{L}{(\mathbb{A}_ne_{2^n})}=\mathscr{A}_{R}{(\mathbb{A}_n)}e_{2^n}.
	 \end{eqnarray}
	 Combining  \eqref{eqpf:cd2} and \eqref{eq:huaar}, we conclude that
	 \begin{eqnarray}\label{eqpf:cd3}
	 \mathscr{A} {(\mathbb{A}_ne_{2^n})}=\mathscr{A}^-({\mathbb{A}_n})e_{2^n},\quad \mathscr{A}^-{(\mathbb{A}_ne_{2^n})}=\mathscr{A}
{(\mathbb{A}_n)}e_{2^n}.
	 \end{eqnarray}
	 By induction, as left modules $$\mathbb{A}_n\cong \O^{2^{n-4}}\oplus \overline{\O}^{2^{n-4}}.$$
	 This implies that
	 $$\mathbb{A}_ne_{2^n}\cong \O^{2^{n-4}}\oplus \overline{\O}^{2^{n-4}}.$$
	 Hence we obtain
	 \begin{eqnarray*}
	 \mathbb{A}_{n+1}&=&\mathbb{A}_n\oplus \mathbb{A}_ne_{2^n}\\
	 &\cong& \O^{2^{n+1-4}}\oplus \overline{\O}^{2^{n+1-4}}.
	 \end{eqnarray*}
	 This completes the proof.
\end{proof}

We finally define the $\O$-inner product on $\mathbb{A}_n$. We    denote by $$\pi_{\O}:\mathbb{A}_n\to \O$$ the  projection.
For any $s,w\in \mathbb{A}_n$, we define
\begin{eqnarray}
\fx{s}{w}_n:=\pi_{\O} (s\overline{w})
\end{eqnarray}

\begin{thm}\label{thm:cd-Hilbert-270}
$(\mathbb{A}_n,\fx{\cdot}{\cdot}_n)$ is a Hilbert left $\O$-module for any $n\geqslant 4$.	
\end{thm}

\begin{proof}
	We have shown that $\fx{s}{w}_4$ is an $\O$-inner product. Suppose $\fx{s}{w}_{n-1}$ is also an $\O$-inner product on $\mathbb{A}_{n-1}$. We denote by $$\pi_{n-1}:\mathbb{A}_n\to \mathbb{A}_{n-1}$$ the orthogonal projection.
	Then for any  $s=a+be_{2^{n-1}},w=c+de_{2^{n-1}}\in \mathbb{A}_{n}$, we have
	\begin{eqnarray*}
		\fx{s}{w}_n&=&\pi_{\O}(s\overline{w})\\
		&=&\pi_3\pi_4\dots \pi_{n-1}(s\overline{w})\\
		&=&\pi_3\pi_4\dots \pi_{n-2}(a\overline{c}+\overline{d}b)\\
		&=&\fx{a}{c}_{n-1}+\fx{\overline{d}}{\overline{b}}_{n-1}.
	\end{eqnarray*}
	Note that for any $p\in \O$, $$ps=pa+(bp)e_{2^{n-1}}.$$
	Hence we have
	\begin{eqnarray*}
		\fx{ps}{w}_n-p\fx{s}{w}_n&=&\fx{pa}{c}_{n-1}+
\big\langle{\overline{d}}, \ \ {\overline{(bp)}}_{n-1}\big\rangle-p\big(\fx{a}{c}_{n-1}+\fx{\overline{d}}
{\overline{b}}_{n-1}\big)\\
		&=&[p,a,c]_{n-1}+[\fx{\overline{d}}{\overline{b}}_{n-1},p]-[p,\overline{d},\overline{b}]_{n-1}.
	\end{eqnarray*}
	Here $[p,a,c]_{n-1}$ and $[p,\overline{d},\overline{b}]_{n-1}$ are second associators in the $\O$-Hilbert space $\mathbb{A}_{n-1}$. Hence by induction we obtain
	$$\re\left(\fx{ps}{w}_n-p\fx{s}{w}_n\right)=0.$$
	This proves the para-linearity.

	To prove the hermiticity,
	we apply the inductive  assumption to get
	$$\fx{w}{s}_n=\fx{c}{a}_{n-1}+\fx{\overline{b}}{\overline{d}}_{n-1}=\overline{\fx{a}{c}_{n-1}}+\overline{\fx{\overline{d}}{\overline{b}}_{n-1}}=\overline{\fx{s}{w}_n}.$$
	
	At last, we come to prove the   positivity.
	According to the inductive  assumption, we have
	$$\fx{s}{s}_n=\fx{a}{a}_{n-1}+\fx{\overline{b}}{\overline{b}}_{n-1}\geqslant 0$$ and $$\fx{s}{s}_n=0\iff a=b=0\iff s=0.$$
	This shows the positivity.
\end{proof}

Theorem \ref{thm:cd-Hilbert-270} provides us a new approach to study the Cayley–Dickson algebras. We point out that \eqref{eqpf:cd3}  gives a complete description  of the associative and conjugate associative elements of  $\mathbb{A}_{n}$ by induction. More precisely,
 \begin{eqnarray*}
	\mathscr{A} {(\mathbb{A}_{n+1})}&=& \mathscr{A} (\mathbb{A}_n )\oplus\mathscr{A}^-({\mathbb{A}_n})e_{2^n},\\
	\mathscr{A}^- {(\mathbb{A}_{n+1})}&=& \mathscr{A}^- (\mathbb{A}_n )\oplus\mathscr{A}({\mathbb{A}_n})e_{2^n}.
 	 \end{eqnarray*}

For example, we have  	
\bigskip

\begin{center}
\renewcommand\arraystretch{2}
	\begin{tabular}{|c|c|c|}
		\hline $\mathbb{A}_{n}$&a basis of $\mathscr{A}{(\mathbb{A}_n)}$&a basis of $\mathscr{A}{(\mathbb{A}_n)}$\\
		\hline $\mathbb{A}_{4}$&$e_0$&$e_8$\\
		\hline $\mathbb{A}_{5}$&$e_0,\ {e_8}e_{16}$&$e_8,e_{16}$\\
		\hline $\mathbb{A}_{6}$&$e_0,\ {e_8}e_{16},\ {e_8}e_{32},\ {e_{16}}e_{32}$&$e_8,e_{16}, e_{32},\  {({e_8}e_{16})}e_{32}$\\
		\hline
	\end{tabular}
\end{center}
\bigskip


Let us look at a weak associative orthonormal basis in the Cayley-Dickson algebra as a Hilbert  left $\O$-module.

In the specific case $n=5$, the algebra  $\mathbb{A}_{n}=\mathbb{A}_{5}$ admits a  weak associative orthonormal   basis $$S:=\{e_0,\ e_{24},e_8,e_{16}\}.$$ This means  for any $x\in \mathbb{A}_{5}$,
$$x=\pi_{\O}(x)+\pi_{\O}(x\overline{e_{8}})e_{8}+\pi_{\O}(x\overline{e_{16}})e_{16}+\pi_{\O}(x\overline{e_{24}})e_{24}.$$

In the   general $\mathbb{A}_{n}$,  due to \eqref{eq:e2n+j}  we can take a weak associative orthonormal   basis as
$$S:=\{e_{8i}:\mid i=0,\dots, 2^{n-3}-1\}.$$
Any $x\in\mathbb A_n$ can be expressed in terms of this weak associative orthonormal   basis as
\begin{eqnarray}\label{def:exp-basis-271}
x=\sum_{i=0}^{2^{n-3}-1}\pi_{\O}(x\overline{e_{8i}})e_{8i}.\end{eqnarray}
This can be regarded as an octonionic version of the decomposition
\begin{eqnarray} \label{def:exp-basis-278}
x=\sum_{i=0}^{2^{n}-1}\re(x\overline{e_{i}})e_i.
\end{eqnarray}


\begin{rem} We have shown that the Cayley-Dickson algebra $\mathbb{A}_{n}$ is a Hilbert left $\O$-module. However the corresponding result does not hold in the quaternionic case. That is,
		$\mathbb{A}_{n}$ is not a $\mathbb{H}$-vector space in general. This is because the definition of $\mathbb{H}$-vector space (see \cite{ng2007quaternionic}) requires an associative condition
		$$p(qx)=(pq)x$$ for all $p,q\in \mathbb{H}$ and   any element $x$ of that $\mathbb{H}$-vector space. \end{rem}

%


	\bibliographystyle{plain}


\end{document}